\documentclass[11pt,tbtags]{amsart}
\usepackage{amssymb}
\usepackage{amsmath}
\usepackage{amsthm}
\usepackage[swedish, english]{babel}
\usepackage[latin1]{inputenc}
\usepackage{psfrag}
\usepackage{graphicx}
\usepackage[square, comma, numbers, sort&compress]{natbib}
\usepackage{subfigure}
\usepackage{nicefrac}
\usepackage{color}
\usepackage[left=3cm,right=3cm,top=3cm,bottom=3cm]{geometry}

\theoremstyle{plain}
\newcommand{\E}{\mathbb E}

\def\P{{\mathbb P}}

\newenvironment{remark}[1][Remark]{\begin{trivlist}
\item[\hskip \labelsep {\bf Remark}]}{\end{trivlist}}
\newtheorem{theorem}{Theorem}[section]

\newtheorem{corollary}[theorem]{Corollary}
\newtheorem{proposition}[theorem]{Proposition}

\theoremstyle{definition}

\title[With or without replacement?]{With or without replacement?\\Sampling uncertainty in Shepp's urn scheme}
\author[Kristoffer Glover]{Kristoffer Glover}
\address{University of Technology Sydney, P.O. Box 123, Broadway, NSW 2007, Australia.}
\date{\today}

\begin{document}

\maketitle

\begin{abstract}
We introduce a variant of Shepp's classical urn problem in which the optimal stopper does not know whether sampling from the urn is done with or without replacement.
By considering the problem's continuous-time analog, we provide bounds on the value function and in the case of a balanced urn (with an equal number of each ball type), an explicit solution is found.
Surprisingly, the optimal strategy for the balanced urn is the same as in the classical urn problem.
However, the expected value upon stopping is lower due to the additional uncertainty present.
\end{abstract}

\maketitle

\section{Introduction}
Consider the following discrete optimal stopping problem as first described in \cite{S} by Shepp.
An urn initially contains $m$ balls worth $-\$1$ each and $p$ balls worth $+\$1$ each, where $m$ and $p$ are positive integers known \emph{a priori}.
Balls are randomly sampled (one at a time and without replacement) and their value is added to a running total.
Before any draw, the optimal stopper can choose to stop sampling and receive the cumulative sum up to that point.
The goal is to find the stopping rule which maximises the expected payout from a given $(m,p)$-urn.

The urn scheme described above was originally formulated in relation to the classical optimal stopping problem of maximising the average value of a sequence of independent and identically distributed random variables (see \cite{Br,CR,Dv}, among others).
The scheme has also been considered in relation to numerous other problems considered in the subsequent literature.
For example, in \cite{CGK}, the authors consider an extension in which the stopper exhibits risk-aversion (modelled as the limited ability to endure negative fluctuations in the running total).
An extension in which the stopper is able to draw more than one ball at a time is also considered in \cite{DK}.
Related to the current note, \cite{CZLW} (and subsequently \cite{Hu}) consider the urn problem where the composition of balls in the urn is not known with certainty (i.e., where $p+m$ is known but $p$ is not).

The aim of the present note is to introduce a variant of Shepp's urn problem in which the sampling procedure used is not known with certainty.
Specifically, while the result of each draw is observable, we assume that the optimal stopper is uncertain about whether or not the balls are removed from the urn after sampling.
In other words, whether sampling is done \emph{with or without replacement}.
Since the probability of sampling a given ball type is different under the two different sampling procedures, sequentially observing the random draws will reveal statistical information about the true procedure being used.
Hence, we adopt a Bayesian approach and assume the optimal stopper has a prior belief of $\pi$ that the samples are not being replaced.
They then, sequentially, update this belief (via Bayes) after each random draw.
Since the goal is to maximise the expected payout upon stopping, any stopping rule must account for the expected learning that will occur over time.

Shepp demonstrated that the optimal rule for the original problem is of a threshold type.
In particular, denoting by $\mathcal{C}$ the set of all urns with a positive expected value (upon stopping optimally), then $\mathcal{C}=\{(m,p)\,|\,m\leq\beta(p)\}$, where $\beta(p)$ is a sequence of unique constants dependent on $p$ (which must be computed via recursive methods, cf.~\cite{Bo2}). 
It is thus optimal to draw a ball if there are sufficiently many $p$ balls relative to $m$ balls (or sufficiently few $m$ balls relative to $p$ balls).
Intuitively, $\beta(p)>p$ and hence, a ball should not be sampled when the current state of the urn satisfies $p-m\leq p-\beta(p)<0$.
Put differently, the optimal stopper should stop sampling when the running total exceeds some critical (positive) threshold, dependent on the current state of the urn.

Of particularly importance to the current note, Shepp \cite[][p.~1001]{S} also connects the urn problem (when the sampling method was known) with the continuous-time problem of optimally stopping a Brownian bridge.
Specifically, via an appropriate scaling, the running total (cumulative sum) process was shown to converge to a Brownian bridge which starts at zero (at $t=0$) and pins to some location $a$ (at $t=1$).
Importantly, the known constant $a$ depends on the initial values of $m$ and $p$, with $a=\tfrac{m-p}{\sqrt{m+p}}$.
Hence, the sign of the pinning location depends on the relative abundance of $m$- and $p$-balls in the urn.
The continuous-time problem is shown by Shepp \cite{S} to admit a closed-form solution and the optimal stopping strategy found, once more, to be of threshold type---being the first time that the Brownian bridge exceeds some time-dependent boundary, given by $a+\alpha\sqrt{1-t}$ with $\alpha\approx 0.83992$.

Given the success and closed-form nature of such continuous-time approximations, we choose not to tackle the discrete version of our problem directly, instead formulating and solving the continuous-time analog.
In such a setting, uncertainty about the true sampling procedure manifests itself in uncertainty about the drift of the underlying (cumulative sum) process.
In particular, the process is believed to be either a Brownian bridge pinning to $a$ (if sampling is done \emph{without} replacement) or a Brownian motion with drift $a$ (if sampling is done \emph{with} replacement), and the optimal stopper must learn about which it is over time.
Despite this additional uncertainty, we find that the problem has a closed-form solution when $a=0$ and, remarkably, the optimal strategy is found to coincide with the optimal strategy of the classical problem (where the sampling procedure/drift is known with certainty).
The expected payout, however, is lower due to the additional uncertainty present.
When $a\neq 0$, the problem is more complicated and a richer solution structure emerges (with multiple optimal stopping boundaries possible).

This note therefore contributes to the literature on both optimally stopping a Brownian bridge (e.g., \cite{S,F,EW,EV2,BCSY,AM,ES}) and optimal stopping in the presence of incomplete information (e.g., \cite{EL,EV,EV3,Ga,G,JP1,JP2}).
We also note that Brownian bridges have found many applications in the field of finance.
For example, they have been used to model the so-called stock pinning effect (see \cite{AL}), and the dynamics of certain arbitrage opportunities (see \cite{BS1,LL}).
In both settings, the existence of the underlying economic force (creating the pinning) is more often than not uncertain.
Hence, the additional uncertainty considered in this note may find application in more realistic modelling of these market dynamics.

The rest of this note is structured as follows:
We start in Section \ref{sec:connect} by commenting further on the connection between the discrete urn problem and the continuous-time analog.
In Section \ref{sec:form}, we formulate the continuous-time problem, making clear our informational assumptions.
Upper and lower bounds on the value function are presented in Section \ref{sec:a0}, along with the explicit solution to the problem in the case where $a=0$.
We conclude in Section \ref{sec:con} with a brief discussion of the case where $a$ is nonzero.

\section{Connecting the urn problem to Brownian bridges/motion}\label{sec:connect}
1. Let $\epsilon_i$, for $i=1,\ldots, m+p$, denote the results of sampling from a given $(m,p)$-urn, with $\epsilon_i=-1$ for an $m$-ball and $\epsilon_i=1$ for a $p$-ball.
The partial sum after $n$ draws is thus $X_n=\sum_{i=1}^{n}\epsilon_i$, with $X_0=0$.
It is well known that the discrete process $\{X_n\}_{n=0}^{m+p}$ can be approximated as a continuous-time diffusion process if we let $m$ and $p$ tend to infinity in an appropriate way.
The resulting diffusion, however, will depend on whether sampling is done with or without replacement.
Fixing $m$ and $p$, we define, for $0\leq n\leq m+p$ and $n<(m+p)t\leq n+1$,
\begin{equation}\label{eqn:process}
  X_{m,p}(t)=\frac{X_n}{\sqrt{m+p}},\quad 0\leq t\leq 1.
\end{equation}

If sampling is done \emph{without} replacement then for $n=m+p$ (after all balls have been sampled) we have
\begin{equation}
  X_{m,p}(1)=\frac{p-m}{\sqrt{m+p}}=:a.
\end{equation}
Hence, the final value (at $t=1$) is known with certainty to be the constant $a$.
In this case, it is also clear that the samples $\epsilon_i$ are not \emph{iid}.
However, Shepp demonstrated that, if $a$ is fixed, the process $X_{m,p}(t)$ converges in distribution as $p\to\infty$ to a Brownian bridge process pinning to the point $a$ at $t=1$ (see \cite[][p.~1001]{S}).

\vspace{4pt}

2. On the other hand, if sampling is done \emph{with} replacement, then the samples $\epsilon_i$ are \emph{iid}, and the process $X_{m,p}(t)$ in \eqref{eqn:process} can be seen to converge in distribution to a Brownian motion (with drift), via Donsker's theorem.
We also note that \emph{with} replacement, more than $m+p$ balls can be sampled.
Indeed, sampling could continue indefinitely if each sampled ball were replaced.
However, after $m+p$ balls have been sampled the true nature of the sampling procedure will be revealed---since there will either be no balls left or another sample is produced.
In our modified urn problem we therefore make the natural assumption that stopping must occur before more than $m+p$ balls have been sampled.\footnote{If the optimal stopper were allowed to continue beyond $m+p$ samples, then the stopper would never stop for $p>m$ ($a>0$) in the sampling-with-replacement scenario (since the cumulative sum is expected to increase indefinitely). It would then also be optimal never to stop \emph{before} $m+p$ balls are sampled (for $\pi<1$ at least) due to the unbounded payoff expected after $m+p$ balls. On the other hand, if $p\leq m$ ($a\leq 0$), the stopper would stop at $m+p$ balls in all scenarios since the cumulative sum process after $m+p$ balls is a supermartingale regardless of whether sampling with or without replacement had been revealed. Thus, the solution to the stopping problems with restricted and unrestricted stopping times would coincide for the case where $a\leq 0$. To avoid the degeneracy of the problem for $a>0$, i.e.~to guarantee a finite value, we chose to restrict the set of admissible stopping times to $n\leq m+p$.}

To apply Donsker's theorem, we note that the probability of drawing a given ball type (with replacement) is constant and given by $p/(m+p)$ for a positive ball and $m/(m+p)$ for a negative ball.
Therefore, $\E[\epsilon_i]=(p-m)/(m+p)=a/\sqrt{m+p}$ and $\textrm{Var}(\epsilon_i)=1-a^2/(m+p)$.
This allows us to rewrite \eqref{eqn:process} as
\begin{equation}\label{eqn:donsk}
  X_{m,p}(t)=\frac{an}{m+p}+\sqrt{1-\tfrac{a^2}{m+p}}\left(\frac{\sum_{i=1}^{n}\widehat{\epsilon}_{i}}{\sqrt{m+p}}\right),
\end{equation}
where $\widehat{\epsilon}_i$ are now standardized random variables (with zero mean and unit variance).
Since we are restricting our attention to $n\leq m+p$, we can once more fix $m$ and $p$ and define $n<(m+p)t\leq n+1$.
Letting $p\to\infty$, the process $X_{m,p}(t)$ in \eqref{eqn:donsk} thus converges to
\begin{equation}\label{eqn:BM}
  X_t=at+B_t, \quad 0\leq t\leq 1,
\end{equation}
where $(B_t)_{0\leq t\leq 1}$ is a standard Brownian motion (cf.~\cite{Do}).
Note that the drift in \eqref{eqn:BM} coincides with the pinning point of the Brownian bridge in the case without replacement.

With this necessary connection in place, we now proceed to formulate the continuous-time stopping problem corresponding to our variant of Shepp's urn scheme.


\section{Problem formulation and learning assumptions}\label{sec:form}
1. Let $X=(X_t)_{t\geq 0}$ denote an observable stochastic process that is believed by an optimal stopper to be either a Brownian motion with known drift $a$, or a Brownian bridge that pins to $a$ at $t=1$.
Adopting a Bayesian approach, we also assume that the optimal stopper has an initial belief of $\pi$ that the true process is a Brownian bridge (and hence a belief of $1-\pi$ that it is a Brownian motion).

This information structure can be realised on a probability space $(\Omega,\mathcal{F},\P_{\pi})$ where the probability measure $\P_{\pi}$ has the following structure
\begin{equation}\label{eqn:space}
  \P_{\pi}=(1-\pi)\P_0+\pi\P_1,\quad \textrm{for } \pi\in[0,1],
\end{equation}
where $\P_0$ is the probability measure under which the process $X$ is the Brownian motion and $\P_1$ is the probability measure under which the process $X$ is the Brownian bridge (cf.~\cite[][Chapter VI, Section 21]{PS}).
More formally, we can introduce an unobservable random variable $\theta$ taking values 0 or 1 with probability $1-\pi$ and $\pi$ under $\P_\pi$, respectively.
Thus, the process $X$ solves the following stochastic differential equation
\begin{equation}\label{eqn:bridge}
  dX_t=\bigl[(1-\theta)a+\theta\bigl(\tfrac{a-X_t}{1-t}\bigr)\bigr]dt+dB_t, \quad X_0=0,
\end{equation}
where $B=(B_t)_{t\geq 0}$ is a standard Brownian motion, independent of $\theta$ under $\P_{\pi}$.

%

%

\vspace{4pt}

2. The problem under investigation is to find the optimal stopping strategy that maximises the expected value of $X$ upon stopping, i.e.
\begin{equation}\label{eqn:prob}
  V(\pi)=\sup_{0\leq\tau\leq 1}\E_{\pi}\left[X_{\tau}\right],\quad \textrm{for } \pi\in[0,1].
\end{equation}
Recall that the time horizon of the optimal stopping problem in \eqref{eqn:prob} is set to one, since the uncertainty about the nature of the process is fully revealed at $t=1$ (it either pins to $a$ or it does not).

If the process was known to be a Brownian bridge then it would be evident from \eqref{eqn:prob} that $V\geq a$, since simply waiting until $t=1$ would yield a value of $a$ with certainty.
However, uncertainty about $\theta$ introduces additional uncertainty in the terminal payoff, since the value received at $t=1$ could be less than $a$ if the true process was actually a Brownian motion.

\vspace{4pt}

%

\begin{remark}
The problem described above is related to the problem studied in \cite{EV2}, in which the underlying process is known to be a Brownian bridge, but for which the \emph{location} of the pinning point is unknown.
Specifically, if the process defined in \eqref{eqn:bridge} was a standard Brownian motion then the distribution of its expected location at $t=1$ would be normal, i.e.~$X_1\sim\mathcal{N}(a,1)$.
On the other hand, if the process was a Brownian bridge pinning to $a$ at $t=1$, then the distribution of its expected location at $t=1$ would be a point mass, i.e.~$X_1\sim\delta_a$ (where $\delta_a$ denotes the Dirac delta).
Hence, setting a prior on the location of the pinning point in \cite{EV2} to $\mu=\pi\delta_a+(1-\pi)\mathcal{N}(a,1)$ is equivalent to the problem formulated in this note.
\end{remark}

\vspace{4pt}

3. To account for the uncertainty about $\theta$ in \eqref{eqn:bridge}, we define the \emph{posterior probability process}
\begin{equation}
  \Pi_t:=\P_{\pi}(\theta=1\,|\,\mathcal{F}_t^X),\quad \textrm{for }t\geq 0,
\end{equation}
which represents the belief that the process will pin at $t=1$ and importantly how it is continually updated over time through observations of the process $X$.
To determine the dynamics of the process $\Pi=(\Pi_t)_{t\geq 0}$, we appeal to well-known results from stochastic filtering theory (see \cite[][Theorem 9.1]{LS} or \cite[][Section 2]{JP1}), namely that, for $t\geq 0$,
\begin{align}
  dX_t &= \bigl[(1-\Pi_t)a+\Pi_t\bigl(\tfrac{a-X_t}{1-t}\bigr)\bigr]dt+d\bar{B}_t, \quad X_0=0, \label{eqn:X} \\
  d\Pi_t &= \rho(t,X_t)\Pi_t(1-\Pi_t)d\bar{B}_t, \quad \Pi_0=\pi, \label{eqn:Pi}
\end{align}
where $\bar{B}=(\bar{B}_t)_{t\geq 0}$ is a $\P_{\pi}$-Brownian motion called the \emph{innovation process} and $\rho$ denotes the \emph{signal-to-noise} ratio that is defined as
\begin{equation}
  \rho(t,X_t):=\frac{a-X_t}{1-t}-a.
\end{equation}
While the payoff in \eqref{eqn:prob} is only dependent on $X$ (not $\Pi$), the drift of $X$ in \eqref{eqn:X} contains $\Pi$.
Therefore, at first blush, it would appear that the optimal stopping problem is two-dimensional (in $X$ and $\Pi$).
However, since both $X$ and $\Pi$ are driven by the same Brownian motion ($\bar{B}$), the problem can, in fact, be reduced to only one spacial variable (either $X$ or $\Pi$) by identifying a (time-dependent) mapping between $X_t$ and $\Pi_t$.
In what follows we will formulate the problem in terms of the original process $X$, since this facilitates a more transparent comparison to the case where the process is known to pin with certainty.

\vspace{4pt}

4. To establish the mapping between $X_t$ and $\Pi_t$ we have the following result.

\begin{proposition}\label{prop:mapping}
Given the processes $X=(X_t)_{t\geq 0}$ and $\Pi=(\Pi_t)_{t\geq 0}$ defined by \eqref{eqn:X} and \eqref{eqn:Pi}, respectively, the following identity holds,
\begin{equation}\label{eqn:link}
  \frac{\Pi_t}{1-\Pi_t}=\frac{\pi}{1-\pi}L^a(t,X_t), \,\,\textrm{ with }\,\, L^a(t,x):=\frac{1}{\sqrt{1-t}}\exp\Bigl(-\frac{1}{2}a^2-\frac{(x-at)^{2}}{2(1-t)}\Bigr),
\end{equation}
for $t\in[0,1)$.
\end{proposition}

\begin{proof}
To establish the mapping we take advantage of the fact that both processes are driven by the same Brownian motion and define the process (cf.~Proposition 4 in \cite{JP2})
\begin{equation}\label{eqn:U}
  U_t = \ln\Bigl(\frac{\Pi_t}{1-\Pi_t}\Bigr)+aX_t-\frac{aX_t}{1-t}+\frac{X_t^2}{2(1-t)},
\end{equation}
which, after applying It\^{o}'s formula, is seen to be of bounded variation with dynamics
\begin{equation}\label{eqn:dU}
  dU_t = \frac{1}{2}\Bigl[a^2-\frac{a^2}{(1-t)^2}+\frac{1}{1-t}\Bigr]dt, \,\textrm{ with } U_0=\ln\left(\pi/(1-\pi)\right).
\end{equation}
Thus, $U_t$ can be solved explicitly as
\begin{equation}\label{eqn:U2}
  U_t = \ln\left(\pi/(1-\pi)\right)+\frac{a^2}{2}t-\frac{a^2}{2(1-t)}-\ln\sqrt{1-t},
\end{equation}
and after combining \eqref{eqn:U} and \eqref{eqn:U2}, we obtain the desired result.
\end{proof}


5. To solve the optimal stopping problem in \eqref{eqn:prob}, we will exploit various changes of measure.
In particular from $\P_{\pi}$ to $\P_0$ (under which the process $X$ is a standard Brownian motion with drift $a$) and then from $\P_0$ to $\P_1$ (under which $X$ is a Brownian bridge pinning to $a$).
In order to perform these measure changes, we have the following result that establishes the necessary Radon-Nikodym derivatives (cf.~Lemma 1 in \cite{JP1}).
\begin{proposition}\label{prop:measure}
Let $\P_{\pi,\tau}$ be the restriction of the measure $\P_\pi$ to $\mathcal{F}_{\tau}^X$ for $\pi\in[0,1]$.
We thus have the following:
\begin{equation}\label{eqn:change}
  \emph{(i)}\,\,\, \frac{d\P_{\pi,\tau}}{d\P_{1,\tau}}=\frac{\pi}{\Pi_{\tau}},\quad \emph{(ii)}\,\,\, \frac{d\P_{\pi,\tau}}{d\P_{0,\tau}}=\frac{1-\pi}{1-\Pi_{\tau}}
\end{equation}
and
\begin{equation}\label{eqn:change2}
  \emph{(iii)}\,\,\,\frac{d\P_{1,\tau}}{d\P_{0,\tau}}=\frac{1-\pi}{\pi}\frac{\Pi_{\tau}}{1-\Pi_{\tau}}=L^a(\tau,X_{\tau}),
\end{equation}
for all stopping times $\tau$ of $X$, where $L^a$ is given in \eqref{eqn:link}.
The process in \eqref{eqn:change2} is often referred to as the \emph{likelihood ratio process}.
\end{proposition}

\begin{proof}
A standard rule for Radon-Nikodym derivatives under \eqref{eqn:space} gives
\begin{align}
  \Pi_\tau & = \P_\pi(\theta=1\,|\,\mathcal{F}_\tau^X) \nonumber \\
  &=(1-\pi)\P_0(\theta=1\,|\,\mathcal{F}_\tau^X)\frac{d\P_{0,\tau}}{d\P_{\pi,\tau}}+\pi\P_1(\theta=1\,|\,\mathcal{F}_\tau^X)\frac{d\P_{1,\tau}}{d\P_{\pi,\tau}}=\pi\frac{d\P_{1,\tau}}{d\P_{\pi,\tau}}\label{eqn:rd1}
\end{align}
for any $\tau$ and $\pi$, yielding identity (i).
Similar arguments show that
\begin{align}
  1-\Pi_\tau &= \P_\pi(\theta=0\,|\,\mathcal{F}_\tau^X) \nonumber \\
  &= (1-\pi)\P_0(\theta=0\,|\,\mathcal{F}_\tau^X)\frac{d\P_{0,\tau}}{d\P_{\pi,\tau}}+\pi\P_1(\theta=0\,|\,\mathcal{F}_\tau^X)\frac{d\P_{1,\tau}}{d\P_{\pi,\tau}}=(1-\pi)\frac{d\P_{0,\tau}}{d\P_{\pi,\tau}},\label{eqn:rd2}
\end{align}
yielding (ii).
Using \eqref{eqn:rd1} and \eqref{eqn:rd2} together, and noting \eqref{eqn:link}, yield (iii).
\end{proof}


\vspace{4pt}

6. Next, we embed \eqref{eqn:prob} into a Markovian framework where the process $X$ starts at time $t$ with value $x$.
However, in doing so, we cannot forget that the optimal stopper's learning about the true nature of the underlying process started at time 0 with an initial belief of $\pi$ and with $X_0=0$.
To incorporate this information, we will exploit the mapping in \eqref{eqn:link} to calculate the stopper's updated belief \emph{should} the process reach $x$ at time $t$.
In other words, in our Markovian embedding, we must assume that the `initial' belief at time $t$ is not $\pi$ but $\Pi_t$ (which depends on $t$ and $x$).
More formally, the embedded optimal stopping problem becomes
\begin{equation}\label{eqn:prob2}
  V(t,x,\pi)=\sup_{0\leq\tau\leq 1-t}\E_{\pi}\bigl[X_{t+\tau}^{t,x}\bigr],
\end{equation}
where the processes $X=X^{t,x}$ and $\Pi$ are defined by
\begin{eqnarray}
  \left\{\begin{array}{ll}\label{eqn:Xembed}
  dX_{t+s}=\bigl(a+\Pi_{t+s}\rho(t+s,X_{t+s})\bigr)ds+d\bar{B}_{t+s}, & 0\leq s<1-t,\\
  X_t=x, & x\in\mathbb{R},\\
\end{array}\right.
\end{eqnarray}
and
\begin{eqnarray}
  \left\{\begin{array}{ll}\label{eqn:Piembed}
  d\Pi_{t+s}=\rho(t+s,X_{t+s})\Pi_{t+s}(1-\Pi_{t+s})d\bar{B}_{t+s}, & 0\leq s<1-t,\\
  \Pi_t=\tfrac{\pi}{1-\pi}L^a(t,x)/\bigl(1+\tfrac{\pi}{1-\pi}L^a(t,x)\bigr)=:\Pi(t,x,\pi), &\\
\end{array}\right.
\end{eqnarray}
respectively.
Note that the function $L^a$ is defined as in \eqref{eqn:link} and, with a slight abuse of notation, we have defined the function $\Pi(t,x,\pi)$ to be the `initial' value of $\Pi$ in the embedding (dependent on $t$, $x$, and $\pi$).
Note further that, since we are able to replace any dependence on $\Pi_{t+s}$ (for $s>0$) via the mapping in \eqref{eqn:link}, we no longer need to consider the dynamics for $\Pi$ in what follows (only the initial point $\Pi_t$).

\vspace{4pt}

7. Since its value will be used in our subsequent analysis, we conclude this section by reviewing the solution to the classical Brownian bridge problem which is known to pin to $a$ (at $t=1$) with certainty (i.e., when $\pi=1$).
In this case, the stopping problem in \eqref{eqn:prob2} has an explicit solution (cf.~\cite[][p.~175]{EW}) given by
\begin{equation}\label{eqn:classical}
  V^{a}_1(t,x):=\left\{\begin{array}{ll}
  a+\sqrt{2\pi(1-t)}(1-\alpha^2)\exp\Bigl(\frac{(x-a)^2}{2(1-t)}\Bigr)\Phi\Bigl(\frac{x-a}{\sqrt{1-t}}\Bigr), & x<b(t),\\
  x, & x\geq b(t),
\end{array}\right.
\end{equation}
for $t<1$ and $V^a_1(1,a)=a$.
The function $\Phi(y)$ denotes the standard cumulative normal distribution function and $b(t):=a+\alpha\sqrt{1-t}$ with $\alpha$ being the unique positive solution to
\begin{equation}\label{eqn:B}
  \sqrt{2\pi}(1-\alpha^2)e^{\frac{1}{2}\alpha^2}\Phi(\alpha)=\alpha,
\end{equation}
which is approximately 0.839924.
(Note that $\pi$ in \eqref{eqn:classical} and \eqref{eqn:B} denotes the universal constant and not the initial belief.)
Further, the optimal stopping strategy in this case is given by
\begin{equation}\label{eqn:taustar}
  \tau_b=\inf\{s\geq 0\,|\,X_{t+s}\geq b(t+s)\},\quad \textrm{for all } t<1.
\end{equation}

\section{Bounds on the value function and solution when $a=0$}\label{sec:a0}

1. As may be expected, the solution to \eqref{eqn:prob2} depends crucially on the value of $a$.
In fact, we find below that the problem is completely solvable in closed form when $a=0$ (corresponding to $m=p$).
For a nonzero value of $a$, the problem is more complicated and a richer solution structure emerges.
However, we are able to provide the following useful bounds on the value function in \eqref{eqn:prob2} for an arbitrary $a$.
Moreover, these bounds can be seen to coincide when $a=0$, yielding the explicit solution in this case.

\begin{proposition}\label{prop:upperbound} \emph{(Upper bound)}.
The value function defined in \eqref{eqn:prob2} satisfies
\begin{equation}\label{eqn:upper}
  V(t,x,\pi)\leq \bigl(1-\Pi(t,x,\pi)\bigr)\bigl(x+\max(a,0)\bigr)+\Pi(t,x,\pi)V_{1}^{a}(t,x),
\end{equation}
where $V_1^a$ is as given in \eqref{eqn:classical} and the function $\Pi$ is the updated belief conditional on the process reaching $x$ at time $t$, defined in \eqref{eqn:Piembed}.
\end{proposition}
\begin{proof}
To establish the upper bound, we consider a situation in which the true nature of the process (i.e., $\theta$) was revealed to the optimal stopper immediately after starting, i.e.~at time $t+$.
In this situation, the optimal stopper would subsequently be able to employ the optimal stopping strategy for the problem given full knowledge of the nature of the underlying process.
Specifically, if the process was revealed as a Brownian bridge, then using $\tau_b$, as defined in \eqref{eqn:taustar}, would be optimal, generating an expected value (at $t=t+$) of $V_1^a(t,x)$.
On the other hand, if the process was revealed as a Brownian motion with drift $a$, then the optimal strategy would be different.
In the case where $a<0$, it would be optimal to stop immediately and receive the value $x$, and in the case where $a>0$ it would be optimal to wait until $t=1$ and receive the expected value $\E_0[X_1]=x+a$.
When $a=0$, however, any stopping rule would yield an expected value of $x$, due to the martingality of the process $X$ in this case.

Considering now the value function at $t=t-$.
Acknowledging that the true nature of the process will be immanently revealed, the expected payout is given by $(1-\Pi_t)(x+\max(a,0))+\Pi_t V_1^a(t,x)$, upon noting that $\Pi_t=\Pi(t,x,\pi)$ represents the current belief about the true value of $\theta$.
Finally, recognizing that the set of stopping times in \eqref{eqn:prob2} is a subset of the stopping times used in the situation described above (where $\theta$ is revealed at $t+$), the stated inequality is clear.
\end{proof}

%

\vspace{4pt}

\begin{proposition}\label{prop:lowerbound} \emph{(Lower bound)}.
The value function defined in \eqref{eqn:prob2} satisfies
\begin{equation}\label{eqn:lower}
  V(t,x,\pi)\geq \bigl(1-\Pi(t,x,\pi)\bigr)\E_0[X_{(t+\tau_b)\wedge 1}]+\Pi(t,x,\pi)V^{a}_{1}(t,x),
\end{equation}
where $V_1^a$ is as given in \eqref{eqn:classical} and $\tau_b$ denotes the optimal strategy for the known pinning case described in \eqref{eqn:taustar}.
Moreover, the function $\Pi$ is the updated belief conditional on the process reaching $x$ at time $t$, defined in \eqref{eqn:Piembed}.
\end{proposition}
\begin{proof}
The desired bound can be established by employing the optimal strategy for the known pinning case, defined in \eqref{eqn:taustar}, in the stopping problem in \eqref{eqn:prob2}, for $\pi<1$.
In detail, letting $X=X^{t,x}$ for ease of notation, we have
\begin{align}
  V(t,x,\pi) &= \sup_{0\leq\tau\leq 1-t}\E_{\pi}\bigl[X_{t+\tau}\bigr]=\sup_{0\leq\tau\leq 1-t}\bigl\{(1-\Pi_t)\E_0\bigl[\tfrac{X_{t+\tau}}{1-\Pi_{t+\tau}}\bigr]\bigr\} \nonumber \\
  &=\sup_{0\leq\tau\leq 1-t}\bigl\{(1-\Pi_t)\E_0[X_{t+\tau}\bigl(1+\tfrac{\Pi_{t+\tau}}{1-\Pi_{t+\tau}}\bigr)\bigr]\bigr\} \nonumber \\
  &=\sup_{0\leq\tau\leq 1-t}\bigl\{(1-\Pi_t)\E_0[X_{t+\tau}]+\Pi_t\E_1[X_{t+\tau}]\bigr\}, \nonumber
\end{align}
where we have applied the measure change from $\P_{\pi}$ to $\P_{0}$, via \eqref{eqn:change}, in the second equality, and the measure change from $\P_0$ to $\P_1$, via \eqref{eqn:change2}, in the last equality.
Furthermore, employing the stopping rule $\tau_b$ from \eqref{eqn:taustar} (which may or may not be optimal) yields
\begin{equation*}
  V(t,x,\pi) \geq (1-\Pi_t)\E_0[X_{(t+\tau_b)\wedge 1}]+\Pi_t\E_1[X_{t+\tau_b}]=(1-\Pi_t)\E_0[X_{(t+\tau_b)\wedge 1}]+\Pi_t V_1^a(t,x),
\end{equation*}
upon noting the definition of $V_1^a$, and where we have ensured that stopping under $\P_0$ happens at or before $t=1$ (since the boundary $b$ is not guaranteed to be hit by a Brownian motion with drift, unlike the Brownian bridge).
\end{proof}

Computation of $\E_0[X_{(t+\tau_b)\wedge 1}]$ is difficult in general, being the expected hitting level of a Brownian motion with drift to a square-root boundary.
Alternatively, we have $\E_0[X_{(t+\tau_b)\wedge 1}]=x+a\E_0[\tau_b\wedge(1-t)]+\E_0[B_{(t+\tau_b)\wedge 1}]=x+a\E_0[\tau_b\wedge(1-t)]$, with the first-passage time $\tau_b=\inf\{s\geq 0\,|\,B_s\geq c(s)\}$, where $c(s):=a(1-s)-x+\alpha\sqrt{1-t-s}$.
Hence, the computation reduces to the problem of finding the mean first-passage time of a driftless Brownian motion (started at zero) to a time-dependent boundary (which is a mixture of a linear and square-root function).
While no explicit expression for $\E_0[\tau_b\wedge(1-t)]$ exists, there are numerous numerical approximations available---see, for example, \cite{DW}, or more recently \cite{HT}. 
When $a=0$, it is clear that $\E_0[X_{(t+\tau_b)\wedge 1}]=x$, a result which we will exploit below.

\vspace{4pt}

2. Given Propositions \ref{prop:upperbound} and \ref{prop:lowerbound}, the following result is evident, and constitutes the main result of this note.
\begin{theorem}\label{thm:solution}
When $a=0$, the value function in \eqref{eqn:prob2} is given by
\begin{equation}\label{eqn:solution}
  V(t,x,\pi)=\bigl(1-\Pi(t,x,\pi)\bigr)x+\Pi(t,x,\pi)V^0_1(t,x),\quad \textrm{for }\pi\in[0,1],
\end{equation}
where $\Pi$ is defined in \eqref{eqn:Piembed} and $V^0_1$ is defined in \eqref{eqn:classical} (upon setting $a=0$).
Further, the optimal stopping strategy in \eqref{eqn:prob2} is given by $\tau^*=\tau_b\wedge(1-t)$.
This stopping strategy is the same for all $\pi\in[0,1]$.
\end{theorem}

\begin{proof}
The result is evident given the fact that the upper bound defined in \eqref{eqn:upper} and the lower bound defined in \eqref{eqn:lower} coincide when $a=0$.
Specifically, we observe that $\E_0[X_{(t+\tau_b)\wedge 1}]=x$ in \eqref{eqn:lower} since $X$ is a $\P_0$-martingale when $a=0$.
Moreover, since the process is not guaranteed to pin at $t=1$, we specify explicitly that the stopper must stop at $t=1$ should the boundary $b$ not be hit.
\end{proof}
Note that the optimality of the solution presented in \eqref{eqn:solution} does not need to be verified since it follows directly from the proven identity in \eqref{eqn:solution} and the existing verification arguments establishing the optimality of $V_1^0$ (provided in \cite{EW}, for example).

The equality found in \eqref{eqn:solution} also demonstrates that (when $a=0$) there is no loss in value due to the optimal stopper using a sub-optimal stopping strategy for the `true' drift. The optimal stopping strategy for a Brownian bridge also achieves the maximum possible value for the Brownian motion (due to martingality), hence using $\tau^*$ will achieve the maximum possible value regardless of the true nature of the underlying process.\footnote{For $a\neq 0$, it would not be possible to achieve the maximum value in both drift scenarios simultaneously through a single optimal stopping rule. Hence there would be loss in value due to this, as indicated by the inequality in Proposition \ref{prop:upperbound}.}

\vspace{4pt}

\begin{remark}
It is also worth noting that the arguments in the proof of Theorem \ref{thm:solution} would carry over to a more general setting in which the process is believed to be either a martingale $M$ or a diffusion $X$ (with an initial probability $\pi$ of being $X$).
In this case, similar arguments to Proposition \ref{prop:upperbound} will show that $V(t,x,\pi)\leq(1-\Pi_t)x+\Pi_tV_1(t,x)$, where $V_1$ denotes the solution to the associated stopping problem for the diffusion $X$.
Under $\P_0$, all stopping rules generate the expected value of $x$, due to $M$ being a $\P_0$-martingale.
Moreover, similar arguments to Proposition \ref{prop:lowerbound} will show that $V(t,x,\pi)\geq(1-\Pi_t)x+\Pi_t V_1(t,x)$, upon using the optimal strategy for the optimal stopping problem under $\P_1$, and noting again that $\E_0[X_{t+\tau}]=x$, for any stopping rule.
Finally, we must note that the function $\Pi_t$ would need to be found on a case-by-case basis via a mapping similar to \eqref{eqn:link}.
In general, however, this mapping could also include path-dependent functionals of the process over $[0,t]$, in addition to the values of $t$ and $x$ (cf.~\cite[][Proposition 4]{JP1}).
\end{remark}

\vspace{4pt}

3. Next, Theorem \ref{thm:solution} also implies the following result.

\begin{corollary}\label{corr:inequal}
When $a=0$, we have $x\leq V(t,x,\pi)\leq V_1^0(t,x)$ and $\pi\mapsto V(t,x,\pi)$ is increasing, with $V(t,x,0)=x$ and $V(t,x,1)=V_1^0(t,x)$.
\end{corollary}

\begin{proof}
From \eqref{eqn:solution} we have that $V-V_1^0=(1-\Pi)(x-V_1^0)\leq 0$ where the inequality is due to the fact that $\Pi\leq 1$ and $V_1^0\geq x$, from \eqref{eqn:classical}.
Direct differentiation of \eqref{eqn:solution}, upon noting \eqref{eqn:Piembed}, also shows that $\tfrac{\partial V}{\partial\pi}=L^0(V_1^0-x)/[(1-\pi)(1+\tfrac{\pi}{1-\pi}L^0)]^2\geq 0$, proving the second claim.
\end{proof}

\vspace{4pt}

Corollary \ref{corr:inequal} reveals that, while the optimal stopping strategy is the same with pinning certainty or uncertainty when $a=0$, the value function with uncertainty is \emph{lower} than that if the pinning was certain/known.
In other words, when sampling from a balanced urn with uncertainty about replacement, the optimal stopping strategy is the same as with replacement, but the expected payout is lower.
To illustrate this, Figure \ref{fig:uncertain} plots the value function $V$ in \eqref{eqn:solution} in comparison to $V_1^0$ as defined in \eqref{eqn:classical}.
We confirm that a larger $\pi$ (hence a stronger belief that the process is indeed a Brownian bridge) corresponds to a larger value of $V$.

\begin{figure}[htp!]\centering
  \includegraphics[width=0.5\textwidth]{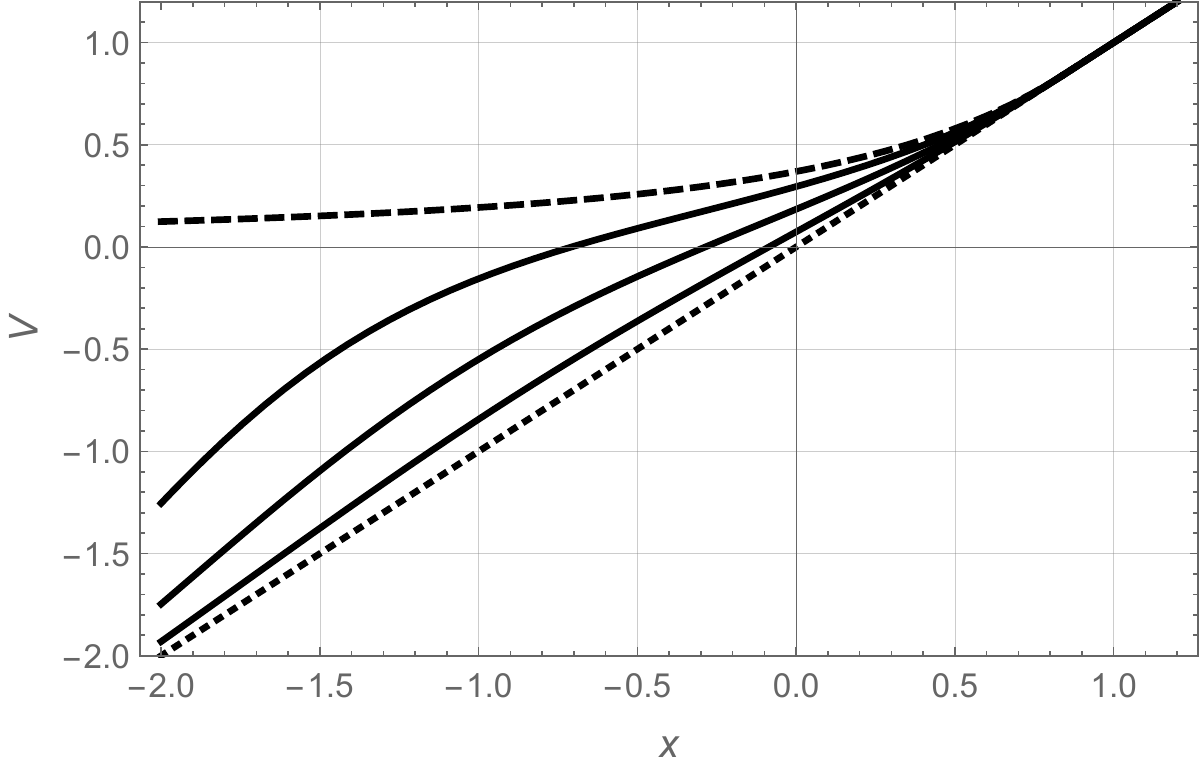}
  \caption{The solution to the problem in \eqref{eqn:prob2} when the process is believed to be a Brownian bridge (pinning to $a=0$) with probability $\pi$ or a (driftless) Brownian motion with probability $1-\pi$. Solid lines = $V(0,x,\pi)$ from \eqref{eqn:solution} for $\pi=\{$0.1, 0.5, 0.9\} (higher lines correspond to larger $\pi$); dashed line = $V_1^0(0,x)$ from \eqref{eqn:classical}; and dotted line = $x$.} \label{fig:uncertain}
\end{figure}

\vspace{4pt}

Figure \ref{fig:uncertain} also highlights the fact that the value function in \eqref{eqn:prob2} can be negative, since pinning to zero is not guaranteed (and hence stopping at $t=1$ does not guarantee a minimum payoff of zero).
For example, if $\pi=0.5$ (i.e., sampling with or without replacement were both initially thought to be equally likely), then the value function in \eqref{eqn:solution} would be negative for all $x<-0.286$.
This does not mean, however, that it would be optimal to stop once the running payoff drops below this value, since an immediate negative payoff would be received, compared to the zero expected payoff from continuing and stopping according to $\tau^*$.

\section{The case where $a$ is nonzero}\label{sec:con}

1. If the urn is not balanced, meaning that $m\neq p$, then a nonzero drift and a nonzero pinning point are introduced into the process $X$.
This asymmetry complicates the problem considerably and, while the bounds in \eqref{eqn:upper} and \eqref{eqn:lower} are still valid, a closed-form solution to \eqref{eqn:prob2} is no longer available.
Attempting to provide a detailed analytical investigation of this case is beyond the scope of this note.
However, numerical investigation of the variational inequality associated with \eqref{eqn:prob2} suggests that a rich solution structure emerges, particularly in the case where $a>0$, when multiple stopping boundaries can arise.
We therefore conclude this note by exposing some of this structure to pique the reader's interest.

\vspace{4pt}

\begin{remark}
It should be noted that if the drift of the Brownian motion was zero, but the Brownian bridge had a nonzero pinning level, then the results of Theorem \ref{thm:solution} would still hold (due to the martingality of $X$ under $\P_0$).
However, this situation does not correspond to the urn problem described in Section \ref{sec:connect}, in which both the drift and the pinning point must be the same.
\end{remark}

\vspace{4pt}

2. To shed some light on the optimal stopping strategy for a nonzero $a$, it is useful to reformulate the problem in \eqref{eqn:prob2} under the measure $\P_0$ as follows:
\begin{align}
  V(t,x,\pi) &= \sup_{0\leq\tau\leq 1-t}\E_{\pi}\bigl[X_{t+\tau}\bigr]=(1-\Pi_t)\sup_{0\leq\tau\leq 1-t}\E_0\bigl[\tfrac{X_{t+\tau}}{1-\Pi_{t+\tau}}\bigr] \nonumber \\
  &=\bigl(1-\Pi(t,x,\pi)\bigr)\sup_{0\leq\tau\leq 1-t}\E_0\bigl[X_{t+\tau}\bigl(1+\tfrac{\pi}{1-\pi}L^a(t+\tau,X_{t+\tau})\bigr)\bigr] \nonumber \\
  &=\bigl(1-\Pi(t,x,\pi)\bigr)\sup_{0\leq\tau\leq 1-t}\E_0[G^{\pi,a}(t+\tau,X_{t+\tau})] \nonumber \\
  &=\bigl(1-\Pi(t,x,\pi)\bigr)\widetilde{V}^{\pi,a}(t,x), \nonumber
\end{align}
where we have used \eqref{eqn:change} in the second equality (to change measure) and the mapping from \eqref{eqn:link} in the third equality (to eliminate $\Pi_{t+\tau}$). We have also defined the auxiliary optimal stoping problem
\begin{equation}\label{eqn:prob3}
  \widetilde{V}^{\pi,a}(t,x):=\sup_{0\leq\tau\leq 1-t}\E_0[G^{\pi,a}(t+\tau,X_{t+\tau})],
\end{equation}
and the payoff function
\begin{equation}\label{eqn:Gpi}
  G^{\pi,a}(t,x):=x\bigl(1+\tfrac{\pi}{1-\pi}L^a(t,x)\bigr),
\end{equation}
where $L^a$ is given in \eqref{eqn:link}, which importantly is dependent on the parameter $a$.

\vspace{4pt}

3. Next, defining the infinitesimal generator associated with $X$ as $\mathbb{L}_X:=\tfrac{1}{2}\tfrac{\partial^2}{\partial x^2}+a\tfrac{\partial}{\partial x}$, then It\^{o}'s formula and an application of the optional sampling theorem for any given $\tau$ yield
\begin{equation}\label{eqn:lagrange}
  \E_0[G^{\pi,a}(t+\tau,X_{t+\tau})]=G^{\pi,a}(t,x)+\E_0\int_0^{\tau}H(t+s,X_{t+s})ds,
\end{equation}
where
\begin{equation}\label{eqn:H}
  H(t,x):=\bigl(\tfrac{\partial}{\partial t}+\mathbb{L}_X\bigr)G^{\pi,a}(t,x)=a-\frac{\pi(x-a)}{(1-\pi)(1-t)}L^a(t,x).
\end{equation}
Hence, from \eqref{eqn:lagrange} it is clear that it would never be optimal to stop at a point $(t,x)$ for which $H(t,x)>0$.
For $a=0$, this region corresponds to $x<0$.
However, the shape of this region is qualitatively different for nonzero $a$.
To illustrate this, Figure \ref{fig:H} plots the behaviour of $H$ for both positive and negative values of $a$.


\begin{figure}[htp!]\centering
  \includegraphics[width=0.45\textwidth]{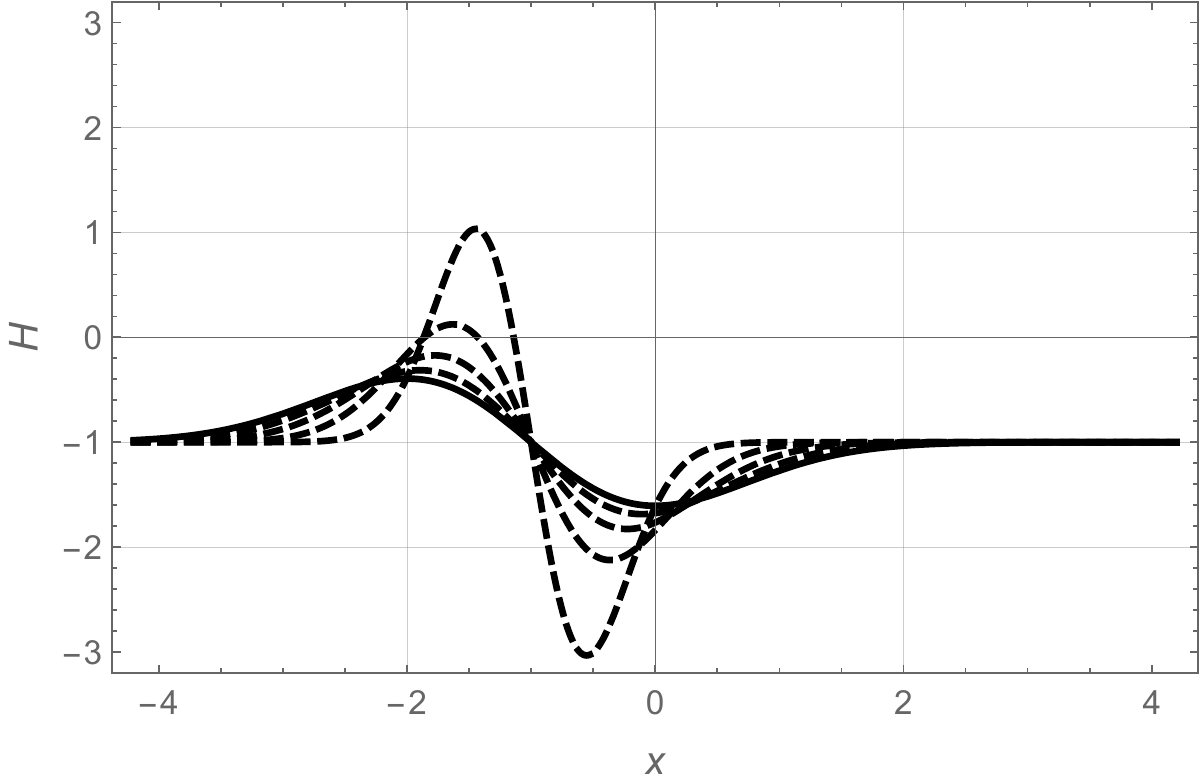}\hspace{3pt}
  \includegraphics[width=0.45\textwidth]{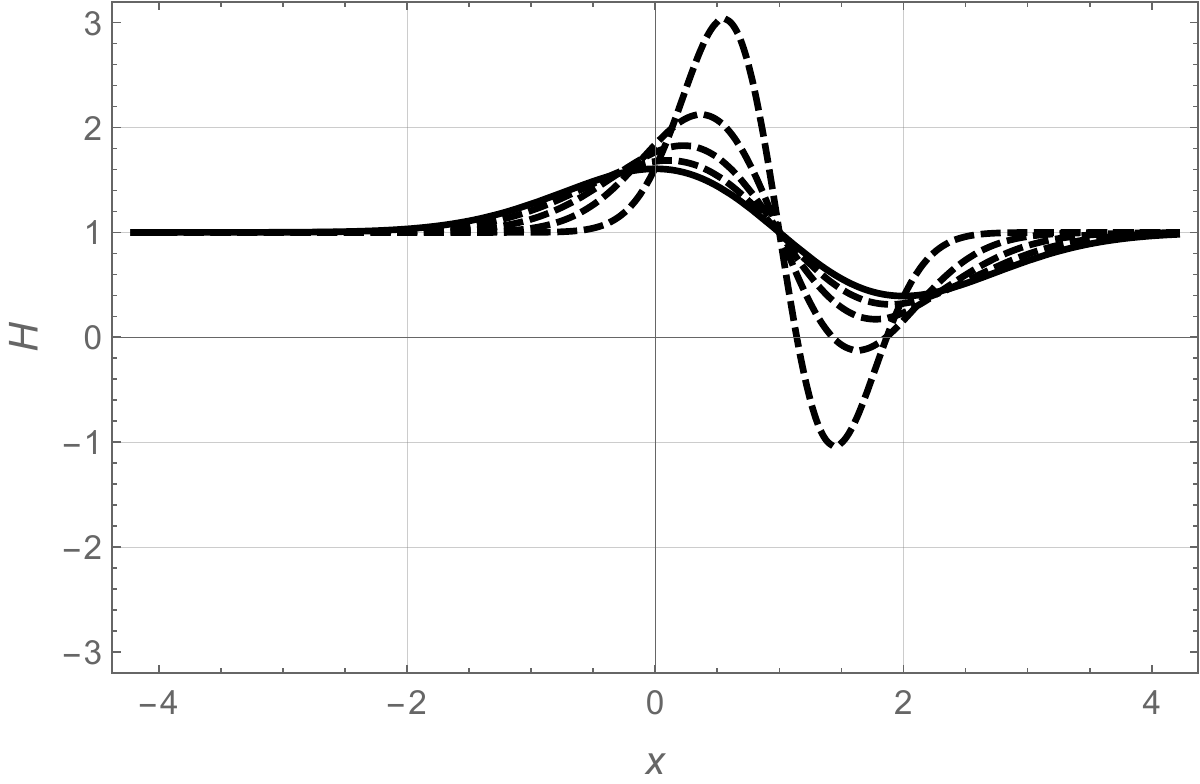}
  \caption{The behaviour of the function $H$ (for $\pi=0.5$ and at various times) for $a=-1$ (on the left) and $a=1$ (on the right). The solid line represents the value at $t=0$ and the dashed lines represent it at $t=\{$0.2, 0.4, 0.6, 0.8\}.} \label{fig:H}
\end{figure}

Considering the case where $a<0$, Figure \ref{fig:H} reveals that $H$ is strictly negative for all $x$ before some critical time (calculated to be 0.536 for the $a=-1$ example).
Furthermore, when the function does become positive, it only does so in a rather narrow interval (below $a$).
This suggests that the incentive to stop is rather strong when $a<0$, as one might expect.
However, little more can be gleaned from the function $H$ in this case.
For $a>0$, however, the function $H$ is more informative about the optimal stopping strategy.
Here, we find that $H$ is strictly positive for all $x$ before some critical time (again found to be 0.536 for $a=1$).
This indicates that when $a>0$, it would never be optimal to stop before this critical time.
Moreover, since $\lim_{x\to\infty}H(t,x)=a$, we also observe that any stopping region must be contained in a finite interval (above $a$).
This suggests the existence of a disjoint continuation region and the presence of two separate optimal stopping boundaries.
Indeed, these predictions are confirmed numerically below.
This richer structure is also consistent with the results in \cite{EV2}, whose authors found similar disjoint continuation regions in a situation where the location of the pinning point of a Brownian bridge was uncertain.

\vspace{4pt}

4. To investigate the solution to \eqref{eqn:prob3}, and hence \eqref{eqn:prob2}, numerically, we employ finite difference techniques applied to an associated variational inequality. The connection between optimal stopping problems and variational inequalities has long been established (see, for example, \citep[][Section 10.4]{O}). Specifically, it can be seen that a candidate solution to \eqref{eqn:prob3} can be obtained by solving the following variational inequality, expressed as a linear complementarity problem (see \cite[][Section 2.5.5]{ZWCS} for a general formulation)
\begin{equation}\label{eqn:LCP}
  \left\{\begin{array}{ll}
  &\Bigl[\tfrac{\partial\widetilde{V}^{\pi,a}}{\partial t}(t,x)+\mathbb{L}_X\widetilde{V}^{\pi,a}(t,x)\Bigr]\Bigl[\widetilde{V}^{\pi,a}(t,x)-G^{\pi,a}(t,x)\Bigr]=0 \\
  &\tfrac{\partial\widetilde{V}^{\pi,a}}{\partial t}(t,x)+\mathbb{L}_X\widetilde{V}^{\pi,a}(t,x) \leq 0 \\
  &\widetilde{V}^{\pi,a}(t,x)-G^{\pi,a}(t,x)\geq 0 \\
  &\widetilde{V}^{\pi,a}(1,x)=G^{\pi,a}(1,x).
  \end{array}\right.
\end{equation}

We have stated the problem as a linear complementarity problem (as opposed to a free-boundary problem with smooth-pasting applied at the unknown boundaries), since the structure of the continuation and stopping regions in $(t,x)$ is not known a priori for \eqref{eqn:prob3}.
As can be seen from \eqref{eqn:LCP}, the location of the optimal stopping boundaries do not appear explicitly in the problem formulation, instead being implicitly defined by the condition $\widetilde{V}^{\pi,a}\geq G^{\pi,a}(t,x)$. Once problem \eqref{eqn:LCP} has been solved (or numerically approximated), the optimal stopping boundaries can simply be read off from the solution by identifying where the function $\widetilde{V}^{\pi,a}-G^{\pi,a}$ switches from being positive to zero. The implicit treatment of the optimal stopping boundaries in \eqref{eqn:LCP} means that any complex structure of the continuation and stopping regions will be revealed as part of the solution. Indeed, this is exactly what we see below when $a>0$, where multiple stopping boundaries are identified.

To approximate the solution to \eqref{eqn:LCP} numerically, we discretise the problem using the Crank-Nicolson differencing scheme and then solve the resulting system of equations using the Projected Successive Over Relaxation (PSOR) algorithm. A more detailed description of the PSOR method, along with proofs of convergence, can be found in \cite{Cr}.\footnote{More details about the discretisation and implementation of the algorithm can also be made available from the author upon request.} Figure \ref{fig:bound} shows the optimal stopping boundaries obtained from our numerical procedure for various values of $a$ (both negative and positive).

\begin{figure}[htp!]\centering
  \includegraphics[width=0.45\textwidth]{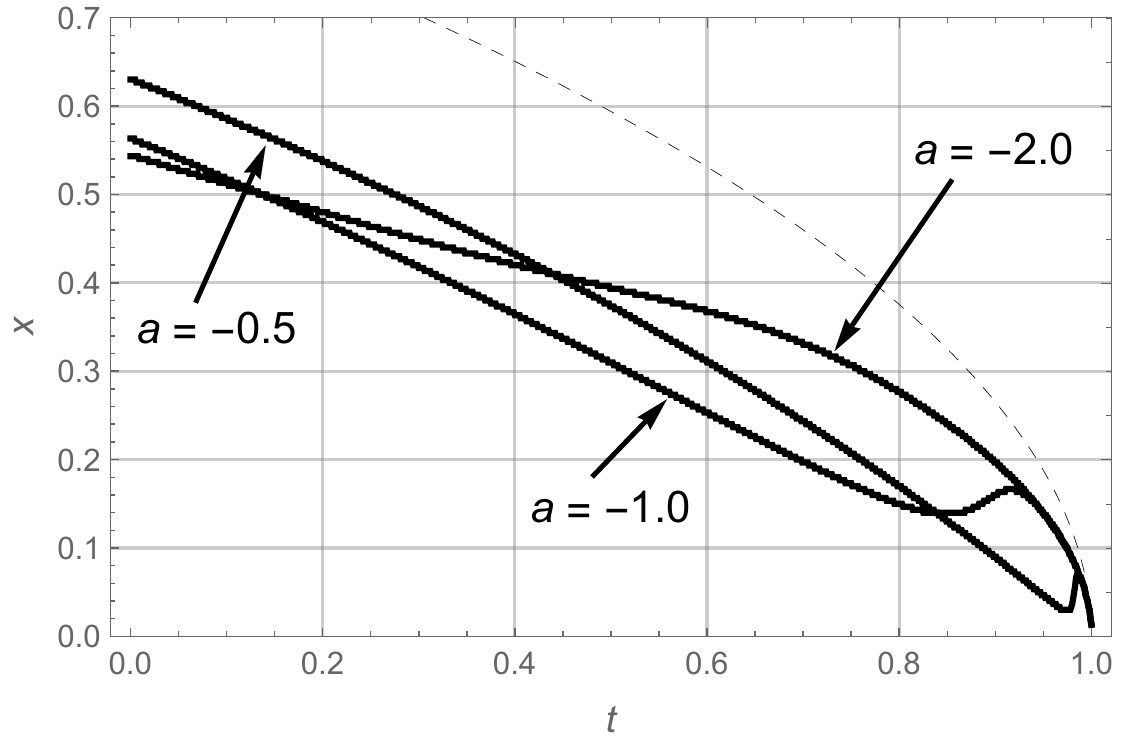}\hspace{3pt}
  \includegraphics[width=0.45\textwidth]{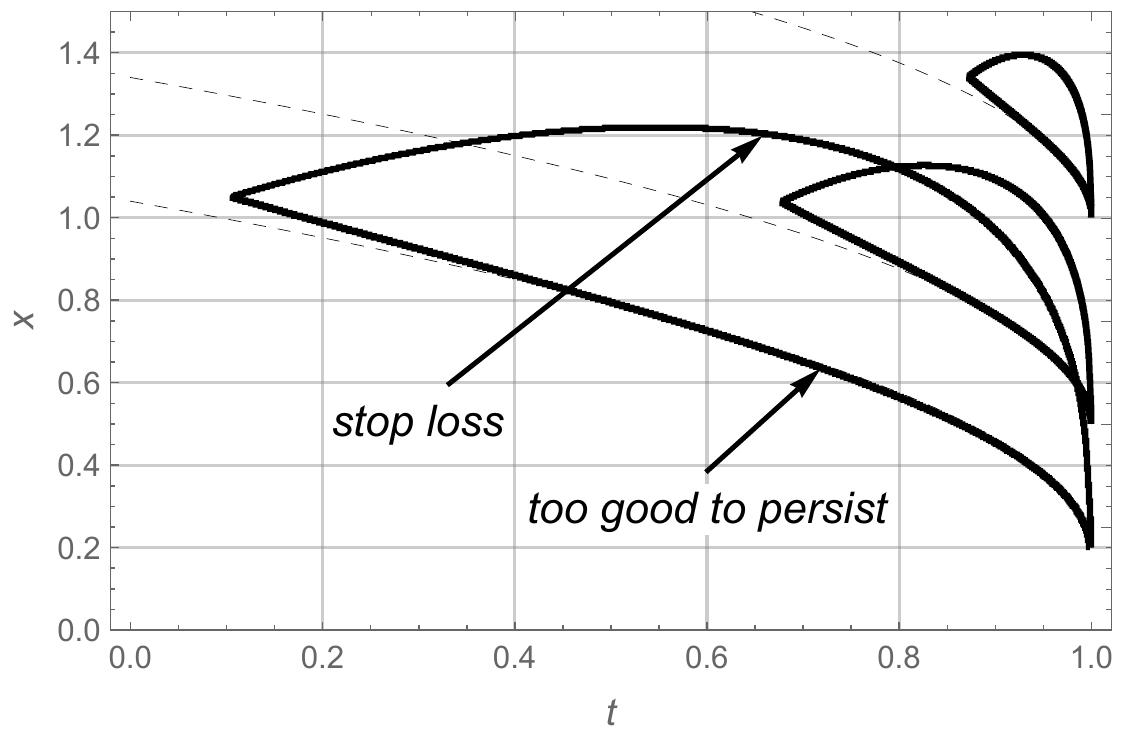}
  \caption{The optimal stopping boundaries (found numerically) for various $a$. On the left: $a=\{$$-$2.0, $-$1.0, $-$0.5\} (solid lines) and $\alpha\sqrt{1-t}$ (dashed line). On the right: $a=\{$0.2, 0.5, 1.0\} (solid lines) and $a+\alpha\sqrt{1-t}$ (dashed lines).} \label{fig:bound}
\end{figure}

Let us first discuss the case where $a>0$.
As predicted, Figure \ref{fig:bound} reveals that it would never be optimal to stop before some critical time (for large enough $a$ or small enough $\pi$ at least).
Recalling that the optimal strategy for a Brownian motion with positive drift is to wait until $t=1$, it would appear that waiting to learn more about the true nature of the process is optimal (at least initially).
In addition, beyond some critical time, we observe two disjoint continuation regions.
Indicating that, depending on the sample path experienced, it can be optimal to stop either after an increase in $X$ (i.e., after a $p$-ball has been drawn) or after a decrease in $X$ (i.e., after an $m$-ball has been drawn).
Based on the terminology introduced in \cite{EV2}, we can interpret the former boundary as a \emph{too-good-to-persist} boundary and the latter as a \emph{stop-loss} boundary.
The emergence of an endogenous stop-loss boundary in the optimal stopping strategy is a unique feature of the problem with uncertain pinning.
Finally, we also observe that both stopping boundaries lie \emph{above} the corresponding boundary if pinning was certain (given by $a+\alpha\sqrt{1-t}$).
Indicating that when $a>0$, stopping will happen later in the presence of pinning uncertainty.

For the case where $a<0$, we have the following remarks.
Firstly, numerical investigations suggest that it is never optimal to stop when $x<0$, despite the negative drift.
Secondly, the optimal stopping strategy appears to be of the form $\tau=\inf\{s\geq 0\,|\,X_{t+s}\geq \widehat{b}(t+s)\}\wedge(1-t)$ for some time-dependent boundary $\widehat{b}$.
Further, $\widehat{b}(t)$ appears to converge to zero at $t=1$, although it does not do so monotonically for all parameters.
Moreover, the boundary itself is not monotone in the parameter $a$, i.e.~$a\mapsto\widehat{b}(t)$ is not monotone.
This behaviour is most likely due to the differing effects of $a$ on the linear drift of the Brownian motion and the pinning behaviour of the Brownian bridge.

Due to the existence of multiple stopping boundaries, and their observed non-monotonic behaviour, further analytical investigation of the problem for $a\neq 0$ would be challenging and is left for the subject of future research.


\begin{thebibliography}{99}
%
\bibitem{AL}
Avellandeda, M. and Lipkin, M. A market-induced mechanism for stock pinning.
{\em Quant. Finance}. 3 (2003), 417--425.
%
%
%
%
%
%
%
%
%
%

\bibitem{Br}
Breiman, L. Stopping rule problems.
{\em Appl. Comb. Math}. (1964), 284--319.


\bibitem{Bo2}
Boyce, W. On a simple optimal stopping problem.
{\em Discrete Math}. 5 (1973), 297--312.

%
%
%
%
\bibitem{BCSY}
Baurdoux, E. J., Chen, N., Surya, B. A., and Yamazaki, K. Optimal double stopping of a Brownian bridge
{\em Adv. in Appl. Probab}. 47 (2015), no.~4, 1212--1234.
%
%
\bibitem{BS1}
Brennan, M. J. and Schwartz, E. S. Arbitrage in stock index futures.
{\em Journal of Business}. 63 (1990), 7--31.
%
%
%
%
%

\bibitem{CGK}
Chen, R. W., Grigorescu, I. and Kang, W. Optimal stopping for Shepp's urn with risk aversion.
{\em Stochastics}. 84 (2015), 702--722.

\bibitem{CZLW}
Chen, R. W., Zame, A., Lin, C. T., and Wu, H. A random version of Shepp's Urn Scheme.
{\em SIAM J. Discrete Math}. 19 (2005), 149--164.

\bibitem{Cr}
Cryer, C. The solution of a quadratic programming problem using systematic overrelaxation.
{\em SIAM J. Control}. 9 (1971), 385--392.
%
%
\bibitem{CR}
Chow, Y.~S. and Robbins, H. On optimal stopping rules for $S_n/n$.
{\em Ill. J. Math}. 9 (1965), 444--454.
%
%

\bibitem{AM}
De Angelis, T. and Milazzo, A. Optimal stopping for the exponential of a Brownian bridge.
{\em J. Appl. Prob}. 57 (2020), no.~3, 361--384.

\bibitem{DK}
Dehghanian, A. and Kharoufeh, J. P. Optimal stopping with a capacity constraint: Generalizing Shepp's urn scheme.
{\em Oper. Res. Lett}. 47 (2019), 311--316.

\bibitem{Do}
Donsker, M. D. An invariance principle for certain probability limit theorems.
{\em Mem. Amer. Math. Soc}. 5 (1951), 43--55.

\bibitem{DW}
Durbin, J. and Williams, D. The first-passage density of the Brownian motion process to a curved boundary.
{\em J. Appl. Prob}. 29 (1992), 291--304.

\bibitem{Dv}
Dvoretzky, A. Existence and properties of certain optimal stopping rules.
{\em Proc. Fifth Berkeley Symp. Math. Statist. Prob}. 1 (1967), 441--452. Univ. of California Press.

%
%
%

\bibitem{ES}
Ernst, P. A. and Shepp, L. A. Revisiting a theorem of L. A. Shepp on optimal stopping.
{\em Commun. Stoch. Anal}. 9 (2015), no.~3, 419--423.

\bibitem{EL}
Ekstr\"om, E. and Lu, B. Optimal selling of an asset under incomplete information.
{\em Int. J. Stoch. Anal}. (2011), Article ID 543590.

\bibitem{EV}
Ekstr\"om, E. and Vaicenavicius, J. Optimal liquidation of an asset under drift uncertainty.
{\em SIAM J. Financ. Math}. 7 (2016), no.~1, 357--381.

%
\bibitem{EV2}
Ekstr\"om, E. and Vaicenavicius, J. Optimal stopping of a Brownian bridge with unknown pinning point.
{\em Stochastic Process. Appl}. 130 (2020), no.~2, 806--823.

\bibitem{EV3}
Ekstr\"om, E. and Vannest\aa l, M. American Options with Incomplete Information.
{\em Int. J. Theor. Appl. Finance}. 22 (2019), no.~6, Article 1950035.

\bibitem{EW}
Ekstr\"om, E. and Wanntorp, H. Optimal stopping of a Brownian bridge.
{\em J. Appl. Prob}. 46 (2009), 170--180.

\bibitem{F}
F\"ollmer, H. Optimal stopping of constrained brownian motion.
{\em J. Appl. Prob}. 9 (1972), 557--571.

\bibitem{G}
Glover, K. Optimally stopping a Brownian bridge with an unknown pinning time: A Bayesian approach.
{\em Stochastic Process. Appl}. Forthcoming (2022).

\bibitem{HT}
Herrmann, S. and Tanr\'{e}, E. The first-passage time of the Brownian motion to a curved boundary: An algorithmic approach.
{\em SIAM J. Sci. Comput}. 38 (2016), A196--A215.

\bibitem{Hu}
Hung, Y.-C. A note on randomized Shepp's urn scheme.
{\em Discrete Math}. 309 (2009), 1749--1759.

\bibitem{Ga}
Gapeev, P. Pricing of perpetual American options in a model with partial information.
{\em Int. J. Theor. Appl. Finance} 15 (2012), no.~1, 1--22.

%
%
%
%
%
%

\bibitem{JP1}
Johnson, P. and Peskir, G. Sequential testing problems for Bessel processes.
{\em Trans. Amer. Math. Soc}. 370 (2018), 2085--2113.

\bibitem{JP2}
Johnson, P. and Peskir, G. Quickest detection problems for Bessel processes.
{\em Ann. Appl. Probab}. 27 (2017), 1003--1056.

%
%
%
%
%
%
%
%
%

\bibitem{LL}
Liu, J. and Longstaff, F. A. Losing money on arbitrage: Optimal dynamic portfolio choice in markets with arbitrage opportunities.
{\em Rev. Financ. Stud}. 17 (2004), 611--641.
%
%
\bibitem{LS}
Liptser, R.S. and Shiryaev, A.N. {\em Statistics of Random Processes I: General Theory (Second Edition)}.
Applications of Mathematics, 5. Springer-Verlag, New York, 2001.
%
%
%
%
%
%
\bibitem{O}
\O{}ksendal, B. {\em Stochastic Differential Equations: An Introduction with Applications (6th Edition)}.
Springer-Verlag, 2004.
%
%
%
\bibitem{PS}
Peskir, G. and Shiryaev, A. N. {\em Optimal Stopping and Free-Boundary Problems}.
Lectures in Mathematics, ETH Z\"{u}rich, Birkh\"{a}user, 2006.
%
%
\bibitem{S}
Shepp, L. A. Explicit solutions to some problems of optimal stopping.
{\em Ann. Math. Statist}. 40 (1969), 993--1010.
%
%
%
\bibitem{ZWCS}
Zhu, Y.-L., Wu, X., Chern, I.-L., and Sun, Z.-Z. {\em Derivative Securities and Difference Methods}.
Springer-Verlag, 2004.



\end{thebibliography}
\end{document}